\documentclass[letterpaper,11pt]{amsart}

\usepackage{geometry}
\usepackage{wrapfig}
\usepackage{caption}
\usepackage{subcaption}
\usepackage[colorinlistoftodos]{todonotes}

\usepackage[utf8]{inputenc}
\usepackage[english]{babel}
\usepackage[T1]{fontenc}
\usepackage{hyperref}
\usepackage{tabularray}

\usepackage{latexsym}
\usepackage{amsfonts}
\usepackage{amssymb}
\usepackage{amsmath}
\usepackage{eucal}
\usepackage[a]{esvect} 
\usepackage{bm} 

\usepackage{amsthm}
\theoremstyle{plain}
\newtheorem{theorem}{Theorem}[section]
\newtheorem{lemma}[theorem]{Lemma}

\theoremstyle{definition}
\newtheorem{definition}[theorem]{Definition}
\newtheorem{example}[theorem]{Example}

\usepackage{tikz}
\usetikzlibrary{automata,positioning,arrows.meta,backgrounds,fit,calc}
\tikzset{arr/.style={-Latex}}
\usepackage{pgfmath}
\usepackage{tikz-cd}

\usepackage[numbers,sort&compress]{natbib}

\usepackage{todonotes}
\setuptodonotes{color=black!10,bordercolor=black!30,textcolor=black}

\newcommand{\cD}{\mathcal{D}}
\newcommand{\cS}{\mathcal{S}}
\newcommand{\cT}{\mathcal{T}}
\newcommand{\cU}{\mathcal{U}}
\newcommand{\cI}{\mathcal{I}}
\newcommand{\cK}{\mathcal{K}}
\newcommand{\Fin}{\mathbf{Fin}} 
\newcommand{\Z}{\mathbb{Z}}

\newcommand{\imgs}[1]{{\mathcal I}#1}
\DeclareMathOperator{\enc}{ enc }
\DeclareMathOperator{\dec}{ dec }

\newcommand{\ENDO}[1]{#1^{\bm{\circlearrowright}}}
\newcommand{\TO}[2]{m_{#1\rightarrow#2}}
\newcommand{\TOREP}[1]{m_{#1\rightarrow\ast}}
\newcommand{\FROMREP}[1]{m_{\ast\rightarrow#1}}

\DeclareMathOperator{\Ob}{Ob}
\DeclareMathOperator{\Hom}{hom}
\DeclareMathOperator{\dom}{dom}
\DeclareMathOperator{\cod}{cod}

\newcommand{\FiniteTransformation}[2]{
\begin{tikzpicture}
  \foreach \i in {1,...,#1} {
    \pgfmathparse{0.3*\i}
    \let\xcoord\pgfmathresult
    \node[draw,thick,circle,inner sep=1pt](t\i) at (\xcoord, 0) {};
    \node[draw,thick,circle,inner sep=1pt](b\i) at (\xcoord, -1) {};
  }
  \foreach \x/\y in #2 {
    \draw[-Latex] (t\x) -- (b\y);
  }
\end{tikzpicture}}

\begin{document}
\title{Representation Independent Decompositions of Computation}

\author{Attila Egri-Nagy$^1$}
\address{$^1$Akita International University, Japan}
\email{egri-nagy@aiu.ac.jp}
\author{Chrystopher L. Nehaniv$^2$}
\address{$^2$University of Waterloo, Canada}
\email{cnehaniv@uwaterloo.ca}

\begin{abstract}
Constructing complex computation from simpler building blocks is a defining problem of computer science.
In algebraic automata theory, we represent computing devices as semigroups.
Accordingly, we use mathematical tools like products and homomorphisms to understand computation through hierarchical decompositions.
To address the shortcomings of some of the existing decomposition methods, \emph{we generalize semigroup representations to semigroupoids by introducing types}.
On the abstraction level of category theory, we describe a flexible, iterative and representation independent algorithm.
Moving from the specific state transition model to the abstract composition of arrows unifies seemingly different decomposition
methods and clarifies the three algorithmic stages:
\emph{collapse, copy} and \emph{compress}.
We collapse some dynamics through a morphism to the top level; copy the forgotten details into the bottom level; and finally we apply compression there.
The hierarchical connections are solely for locating the repeating patterns in the compression.
These theoretical findings pave the way for more precise computer algebra tools and allow for understanding computation with other algebraic structures.
\end{abstract}

\maketitle

  \hypersetup{hidelinks}
  \tableofcontents

\section{Introduction}

Organizing computation into a hierarchical network of modular pieces is pervasive at all levels of software, hardware and communication. An app is built on several libraries; a library is a collection of functions calling other functions and operating system routines. A microprocessor consists of several units,  which also have their internal structures down to the transistors.
Communication protocols, like the Internet protocol suite (TCP/IP), consist of several abstraction layers, where the lower layers do not expose their details.
All these hierarchical decompositions can be modeled in state transition
systems, which have several mathematical descriptions. \emph{What is the most suitable mathematical representation for state transition systems?}
Following the problems of practical usability of decomposition algorithms in algebraic automata theory, we argue  typed semigroups (semigroupoids) are the most natural choice.

Finding the right level of abstraction is an important aspect of writing mathematics and developing software.
When we are too specific, the results have narrow reach and the use cases are few.
If we overdo abstraction, then a price is paid for mathematical text by the cognitive load of the
reader,  and for program execution by the runtime overhead of the layers and interfaces.
It is a fortunate situation when the suitable level of abstraction reveals itself.
This happens in the hierarchical decomposition of transformation semigroups.
The Covering Lemma method \cite{FromRelToEmulationNehaniv1996,egrinagy2024relation} creates a two-level hierarchical decomposition of a transformation semigroup, however the second level does not form a semigroup.
This is an obstacle for iterating the construction and arguably unwieldy to deal with two levels of abstraction at the same time.
In this paper, we generalize decomposition to semigroupoids, where the problem simply does not occur.
In a way, this is one good answer for the question `What is category theory useful for?'.
Therefore, we reformulate the decomposition method on the categorical level.
Admittedly, we make the assumption that a lower number of algebraic structures involved in the explanation is preferred.
This is not merely an aesthetic principle.
Computer algebra implementation can benefit from the simplification: it could improve readability, verifiability and efficiency.

\subsection{Background}
The Krohn-Rhodes Prime Decomposition Theorem \cite{primedecomp65} is a seminal result in \emph{algebraic automata theory}.
Informally speaking, the theorem states that any finite state transition system can be built from smaller, simpler components in a hierarchical way.
As the control information flow is unidirectional in those decompositions, the cognitive operation of abstraction is particularly easy to carry out.
This makes the decomposition theorem central to scientific understanding and promises numerous applications for systems with automata models \cite{wildbook}.
The \textsf{SgpDec} semigroup decomposition package \cite{SgpDec} for the
\textsf{GAP} computer algebra system \cite{GAP4} provides a computational tool,
but it has limitations in scalability and flexibility of the implemented algorithms.
This paper aims to advance the theory in order to enable better decomposition tools.

The natural connection between abstract algebra and automata theory is deepened by \emph{category theory} \cite{MacLane1998,awodey2010category,Leinster_2014,riehl2017category}.
On the \emph{semigroup theory} side, \cite{TilsonCat} uses a derived category to
encode information lost in a surjective semigroup homomorphism.
This is one of the origins of the Covering Lemma method.
The monograph \cite{QBook} shows how far the categorical approach has gone in the theory of
finite semigroups.

From \emph{computer science} perspective, state
transition systems can be generalized as set-valued functors \cite{barr1995category}.
It is also argued there, that automata can be used as another  starting point for defining categories, augmenting the more traditional approach beginning with partial orders.
Generalization goes from single object to many object categories leading to  \emph{typed monoid action}.
The textbook \cite{lawvere2009conceptual} also develops categories as discrete
dynamical systems.
This paper aims to extend and use these ideas for practical decomposition algorithms.

\subsection{Notation}
We assume basic category theory knowledge including categories and functors (e.g., \cite{awodey2010category,Leinster_2014}).
However, we use different notational conventions for composition.
In category theory, the standard notation is $g\circ f$, since it works well
with the usual function application $g(f(X))$.
Here, we act and compose on the right.
We write $Xfg$ and thus the composite is $fg$, just like an automaton would read its input symbols from left to right.
In this paper we study finite structures, and most of the time we omit mentioning finiteness.

\subsection{Structure of the paper} In Section \ref{sect:sgps}, we briefly describe the original Covering Lemma algorithm for transformation semigroups by an example.
We point out the main deficiencies of the existing implementation and give a specification for an improved algorithm.
In Section \ref{sect:sgpoids}, we introduce types to semigroups by giving the definition of semigroupoids and their structure preserving maps.
In Section \ref{sect:sgpoiddec}, we give a representation independent semigroupoid decomposition algorithm that solves the problems described in Section \ref{sect:sgps}.
The main result is Theorem \ref{thm:sgpoiddec}.
In Section \ref{sect:holonomy}, we demonstrate  that the celebrated holonomy decomposition \cite{eilenberg} is essentially the same as the Covering Lemma method.

\section{Covering Lemma Decomposition Method for Transformation Semigroups}
\label{sect:sgps}

Semigroups are algebraic objects, i.e., sets with some added structure defined by operations on the elements.
The abstract definition does not give any details about those elements.

\begin{definition}
A \emph{semigroup} is a set $S$ with an associative binary operation $S\times S\rightarrow S$.
A \emph{monoid} is a semigroup with an identity element $e$, such that $es=se=s$ for all $s\in S$.
\end{definition}

A monoid can be viewed as a category with a single object acting as a
placeholder for the (loop) arrows that compose according to the defining composition table.
In automata theory we are interested in state transition systems, thus we give
meaning to the arrows.
We interpret the semigroup elements as transformations of a state set.
\begin{definition}[Transformation Semigroups]
A \emph{transformation semigroup} $(X,S)$ is a finite nonempty set of
\emph{states} (points) $X$   and a set $S$ of total transformations of $X$, i.e.,
functions of type $X\rightarrow X$, closed under composition.
\end{definition}

From the categorical perspective, a transformation representation of a monoid
$M$ is a set-valued functor $M\rightarrow\Fin$, into the category of finite sets
and functions between them.
\begin{example}[The flip-flop monoid]
\label{ex:flip-flop}
A composition table defines the monoid abstractly (Fig.~\ref{fig:flip-flop}).
It describes equations, e.g., $w_1r= w_1$, $w_1w_0=w_0$.
We can interpret the elements of the monoid: $r$ -- read, $w_0$ -- write 0, $w_1$ -- write 1.
The read operation acts as an identity.
To see why this makes sense, we need to represent the elements as transformations of a set.
Here, we can use the set $\{0,1\}$, representing the two states of a 1-bit memory device.
An example of a computation: $0w_1rrw_0rw_1r$ visits states 0, 1, 1, 1, 0, 0, 1, 1 in succession.
This is untyped, so the monoid elements can be put into any sequence.
\end{example}

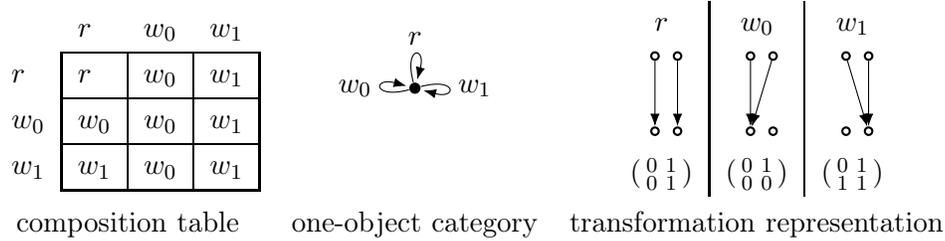
\begin{figure}
\begin{tblr}{ccc}
  \begin{tblr}{
      hline{2-5}={2-5}{0.4pt,solid},
      vline{2-5}={2-5}{0.4pt,solid}}
     &$r$&$w_0$& $w_1$\\
    $r$&$r$&$w_0$&$w_1$\\
    $w_0$&$w_0$&$w_0$&$w_1$\\
    $w_1$&$w_1$&$w_0$&$w_1$
  \end{tblr}
&
\begin{tikzpicture}
  \tikzset{obj/.style={fill,circle,inner sep=1.5pt},
          arr/.style={-Latex}}
  \node [obj] (o) {};
  \path[->,every loop/.append style=arr]
  (o) edge [loop above] node {$r$} (o)
  (o) edge [loop left] node {$w_0$} (o)
  (o) edge [loop right] node {$w_1$} (o);
\end{tikzpicture}
&
  \begin{tblr}{c|c|c}
$r$ & $w_0$ & $w_1$ \\
  \FiniteTransformation{2}{{1/1,2/2}} & \FiniteTransformation{2}{{1/1,2/1}} & \FiniteTransformation{2}{{1/2,2/2}}\\
  $\left(\begin{smallmatrix}
    0 & 1 \\
    0 & 1
  \end{smallmatrix}\right)$ & $\left(\begin{smallmatrix}
    0 & 1 \\
    0 & 0
  \end{smallmatrix}\right)$ & $\left(\begin{smallmatrix}
    0 & 1 \\
    1 & 1
  \end{smallmatrix}\right)$
\end{tblr}\\
composition table & one-object category &transformation representation
\end{tblr}
\caption{Representations of the flip-flop monoid.}
\label{fig:flip-flop}
\end{figure}

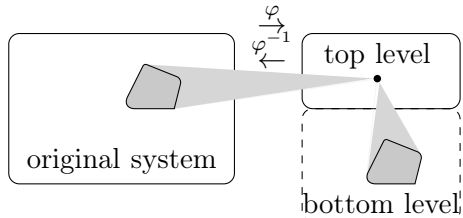
\begin{figure}
  \begin{tikzpicture}
    \newcommand{\myshape}{
      \draw[fill=gray!42,thin] (0,0) [rounded corners=1pt] --
      (.3,.6) [rounded corners=2pt] -- (0.8,.4)
      [sharp corners] -- (0.7,0)
    [rounded corners=5pt] -- cycle;
    }
    \coordinate (X) at (3.9,2);
    \coordinate (pinhole) at ($(X) +(1,-.6)$);
    \coordinate (Y) at ($(X) +(0,-2.5)$);
    \draw[rounded corners] (0,0) rectangle ++(3,2);
    \draw[rounded corners] ($(X) +(0,-1)$) rectangle ++(2.1,1);
    \draw[rounded corners,dashed] (Y) rectangle ++(2.1,1.5);
  \node[inner sep=1pt, circle, fill=black] at (pinhole) (P) {};
   \path[fill=gray!32] (2.2,1) -- (P) -- (1.81,1.59) -- (2.2,1);
   \path[fill=gray!32] ( $(Y) +(0.86,0.57)$ ) -- (P) -- ($(Y) +(1.58,0.89)$) -- ($(Y) +(0.86,0.57)$);
  \node at (1.5,.3) {original system};
  \node at (4.9,1.7) {top level};
  \node at (4.95,-.25) {bottom level};
  \node at (3.5,2.2) {$\overset{\varphi}{\rightarrow}$};
  \node at (3.5,1.8) {$\overset{\varphi^{-1}}{\leftarrow}$};
  \begin{scope}[yshift=1cm,xshift=1.5cm]
    \myshape
  \end{scope}
  \begin{scope}[shift={($(Y) + (0.8,.5)$)}]
    \myshape
  \end{scope}
  \end{tikzpicture}
  \caption{The camera obscura (pinhole camera) metaphor for hierarchical decompositions of state transition systems. The surjective morphism $\varphi$ defines the top level of the decomposition. All the information lost in this map is `projected' down to the second, lower level component. However, the resulting bottom level component is not a well-defined semigroup in the case  $\varphi$ is a morphism of semigroups.}
\label{fig:pinhole}
\end{figure}

Computation is built from irreversible (memory storage, Example \ref{ex:flip-flop}) and reversible parts (permutation groups, Example \ref{ex:Z4}).
How exactly can we build and decompose computing structures is our main interest here.
For transformation semigroups, the Covering Lemma decomposition method \cite{FromRelToEmulationNehaniv1996,egrinagy2024relation} is an easy to explain algorithm as it uses a very simple idea.
We decompose an automaton into two parts, top and bottom.
We identify the top part first, then form the bottom one from what is left out from the top.
This sounds like a vacuously true description of any decomposition.
However, we cannot cut an automaton into two arbitrary halves.
The top and the bottom have to be related hierarchically.
Only special situations allow one to build a system from two independent
parts.
Moreover, there has to be a morphic relationship between the original and the top part.
For semigroups we need to use relations instead of functions for morphisms,
otherwise some semigroups (most notably the full transformation semigroup)
become indecomposable.
\begin{definition}[Relational Morphism]
  A \emph{relational morphism} of transformation semigroups
  $\varphi:(X,S)\rightarrow (Y,T)$ is a pair of relations $(\varphi_0: X\rightarrow Y, \varphi_1: S \rightarrow T)$ that are fully defined, i.e., $\varphi_0(x)\neq\varnothing$ and $\varphi_1(s)\neq\varnothing$, and satisfy the condition of compatible actions
$\varphi_0(x)\cdot\varphi_1(s)\subseteq \varphi_0(x\cdot s)$ for all $x\in X$ and $s\in S$.
The subscript on $\varphi$ indicates the `dimensionality' of the map: $\varphi_0$ for states (0-dimensional points), $\varphi_1$ for state transitions (1-dimensional edges between states).
\end{definition}

The three main steps of the decomposition are as follows.
\begin{enumerate}
  \item We construct an approximate description by a structure-forgetting map
    $\varphi$, which is a surjective relational morphism. The morphism has to \emph{collapse}
    some of the dynamics to be useful. Its image serves as the top level component of the hierarchical decomposition.
\item We recover all the details lost in the first step and package them into the second level component.
Metaphorically speaking, we will use the top level component as a pinhole camera
(see Fig.~\ref{fig:pinhole}).
We pick a state (the pinhole), and peek through  (along $\varphi^{-1}$) to see
the corresponding states and their transformations in the original automaton.
In other words, we simply \emph{copy} the preimages.
\item We identify the projected pieces when there is an isomorphism between them. This last phase is \emph{compression}.
\end{enumerate}

We will recall the algorithm from \cite{egrinagy2024relation} by using a very simple example.
Odometers are hierarchically coupled counters.
Here we build a 4-counter from coupling two 2-counter.

\begin{example}[Building a modulo four counter  $\Z_4$  from two $\Z_2$ counters.] $\Z_4$ is
  the transformation semigroup defined by $(\{0,1,2,3\},\{^+0,^+1,^+2,^+3\})$
  (see Fig.~\ref{fig:Z4}).
  States are denoted by numbers and operations are prefixed by $^+$.
To decompose, first we construct a surjective morphism to $\Z_2$ by
$\varphi_0(0)=\varphi_0(2)=0$, $\varphi_0(1)=\varphi_0(3)=1$ on states. Thus,
$\varphi_0^{-1}(0)=\{0,2\}$ and $\varphi_0^{-1}(1)=\{1,3\}$.
On the projected bottom level we have functions with different domains.
We have two versions of $^+2$, one acting on $\{0,2\}$, the other on $\{1,3\}$.
To distinguish, we can index them by their `pinholes':
$$^+2_0=\left(\begin{matrix}
  0 & 2\\ 2& 0
\end{matrix}\right),
^+2_1=\left(\begin{matrix}
  1 & 3\\ 3& 1
\end{matrix}\right).
$$
\label{ex:Z4}
\end{example}

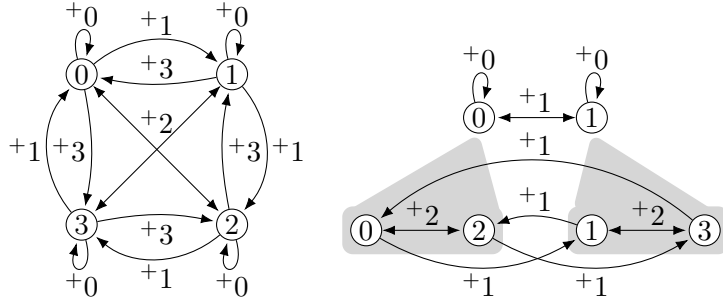
\begin{figure}
\begin{tblr}{cc}
\begin{tikzpicture}[shorten >=1pt, node distance=2cm, on grid, auto,inner sep=2pt]
  \tikzset{arr/.style={-Latex},
  rarr/.style={Latex-Latex},
  bg/.style={fill=gray!32,draw=none,rounded corners,inner sep=10pt},
  n/.style={draw,circle,inner sep=1pt,fill=white}}
  \node[n] (n0)   {$0$};
  \node[n] (n1) [right of=n0] {$1$};
  \node[n] (n2) [below of=n1] {2};
  \node[n] (n3) [below of=n0] {3};
  \path[->,every loop/.append style=arr]
  (n0) edge [arr, loop above] node {$^+0$} (n0)
  (n1) edge [arr, loop above] node {$^+0$} (n1)
  (n2) edge [arr, loop below] node {$^+0$} (n2)
  (n3) edge [arr, loop below] node {$^+0$} (n3)
  (n0) edge [arr,bend left=40] node {$^+1$} (n1)
  (n1) edge [arr,bend left=40] node {$^+1$} (n2)
  (n2) edge [arr,bend left=40] node {$^+1$} (n3)
  (n3) edge [arr,bend left=40] node {$^+1$} (n0)
  (n0) edge [arr,bend left=11,left] node {$^+3$} (n3)
  (n1) edge [arr,bend left=11,above] node {$^+3$} (n0)
  (n2) edge [arr,bend left=11,right] node {$^+3$} (n1)
  (n3) edge [arr,bend left=11,below] node {$^+3$} (n2)
  (n0) edge [rarr,above] node[yshift=4pt] {$^+2$} (n2)
  (n1) edge [rarr]  (n3);
  \end{tikzpicture}
  &
  \begin{tikzpicture}[shorten >=1pt, node distance=1.5cm, on grid, auto,inner sep=2pt]
    \tikzset{arr/.style={-Latex},
    rarr/.style={Latex-Latex},
    bg/.style={fill=gray!32,draw=none,rounded corners,inner sep=10pt},
    n/.style={draw,circle,inner sep=1pt,fill=white}}
      \node[n] (n0)   {$0$};
      \node[n] (n1) [right of=n0] {$1$};
      \node[n,below of=n0] (n22) {$2$};
      \node[n,left of=n22] (n20) {$0$};
      \node[n,below of=n1] (n21) {$1$};
      \node[n,right of=n21] (n23) {$3$};
      \begin{pgfonlayer}{background}
        \fill[bg] ($(n20.north west) +(-3pt,3pt)$) -- (n0.south) -- ($(n22.east) +(3pt,3pt)$) -- cycle;
        \node[fit=(n20) (n22), rounded corners, draw=none, bg, inner sep=3pt] (X) {};
        \fill[bg] ($(n21.north west) +(-3pt,3pt)$) -- (n1.south) -- ($(n23.east) +(3pt,4pt)$) -- cycle;
        \node[fit=(n21) (n23), rounded corners, draw=none, bg, inner sep=3pt] (Y) {};
      \end{pgfonlayer}
      \path[->,every loop/.append style=arr]
        (n0) edge [arr,loop above] node {$^+0$} (n0)
        (n1) edge [arr,loop above] node {$^+0$} (n1)
        (n20) edge [rarr] node {$^+2$} (n22)
        (n21) edge [rarr] node {$^+2$} (n23)
        (n20) edge [arr,bend right=30,below] node {$^+1$} (n21)
        (n22) edge [arr,bend right=30,below] node {$^+1$} (n23)
        (n23) edge [arr,bend right=40,above] node {$^+1$} (n20)
        (n21) edge [arr,bend right=20, above] node {$^+1$} (n22)
        (n1) edge [rarr,above] node {$^+1$} (n0);
    \end{tikzpicture}
\end{tblr}
  \caption{Decomposing $\Z_4$ counter. On the left: the original permutation group. On the right: a two-level decomposition. The bottom level is defined by the pinhole projections. To avoid clutter, the $^+3$ permutation is not shown (the arrows can be obtained by reversing the $^+1$ arrows). The issue is  the bottom level does not form a transformation semigroup.}
  \label{fig:Z4}
\end{figure}
The pinhole projections give fragments of the original semigroup that are not composable.
Consequently, we do not have a semigroup on the second level to compute with.
Moreover, this renders the iteration of the algorithm impossible.
Iterative decompositions would work in several stages.
After separating the top part, we would apply the algorithm again to the bottom component.

\subsection{`Wrong' Ways to Restore Compositionality} We can take the union of the domains and \emph{pad the transformations with identities}.
$$^+2_0=\left(\begin{matrix}
  0 & 2 & 1 & 3\\ 2& 0 & 1 & 3
\end{matrix}\right),
^+2_1=\left(\begin{matrix}
0& 2 &  1 & 3\\ 0 & 2 &3& 1
\end{matrix}\right).
$$
However, this could introduce new elements  are not present in the original semigroup. (This little example does not exhibit this though, since $^+2=\left(\begin{smallmatrix}
  0 & 2 & 1 & 3\\ 2& 0 & 1 & 3
\end{smallmatrix}\right)$ is indeed in $\Z_4$. Example \ref{ex:dual-counter}
exhibits the problem.)

We can also make the transformations \emph{partial} by adding an extra sink
state represented by $\_$.
$$^+2_0=\left(\begin{matrix}
 \_ & 0 & 2 & 1 & 3\\ \_ & 2& 0 & \_ & \_
\end{matrix}\right),
^+2_1=\left(\begin{matrix}
\_ & 0& 2 &  1 & 3\\ \_ & \_  & \_ &3& 1
\end{matrix}\right).
$$
When composing two functions with non-matching codomain and domain, the result will be the unique fully undefined transformation.
This is more like error handling.
Arguably, this extra extension does not correspond to anything in the original semigroup, and may be undesirable (e.g., `crashing' in Example \ref{ex:dual-counter}).

Alternatively, we could use the concept of \emph{variable sets} \cite{lawvere2003sets}.
In contrast with constant sets, they can have different elements based on some parameter, e.g.~time.
Here, that parameter is the state in the top component.
Though, reinterpreting the foundations of mathematics may be unwarranted for
fixing an algorithm, when there are other possibilities.

The existing computer algebra implementation \cite{SgpDec} aims to `reuse' the states by doing the pinhole projections always into the canonical $\{1,\ldots, k\}$ set.
This only provides the illusion of composability.

\subsection{Desiderata for an Iterative Decomposition Algorithm}

We need efficient decompositions (in terms of the number of states and
transformations) and we also require a well-defined bottom level component.

The 4-counter (Example \ref{ex:Z4}) shows  that simple copying alone does not produce a small component.
We expect to have 2 states, not 4 states on the second level.
Here is how we can compress the second level component.

We see  two sets $\{0,2\}$ and $\{1,3\}$ are isomorphic as sets, i.e., there is a bijection between them in the original counter.
Consequently, their permutation groups are isomorphic.
We can choose a \emph{representative}, let's say $\{0,2\}$ in this case.
At the end, we can relabel the states to 0 and 1, to have a proper-looking $\Z_2$ counter.
For now, we work with the copied states.
We can choose the operation $^+1$ for the bijection $0\mapsto 1$, $2\mapsto 3$.
The inverse bijection is $^+3$.
Let's denote them by $f$ and $f^{-1}$.

On the top level we count 1s, on the bottom level 2s.
The bottom level state set is the chosen representative.
How can we say  we are in state 0 or in state 1 in the original $\Z_4$?
We use the top level state.
How do we encode $^+2$ on $\{0,2\}$?
Regardless of the top level state, we should have a transposition in the lower level component.
If the top level state is 0, we simply apply $^+2$ that swaps $0$ and $2$.
If the top level state is 1, we apply $f\,{^+2}\,f^{-1}$, getting the same state transformation.
What happens when we move from $\{0,2\}$ to $\{1,3\}$? For instance, when applying $^+1$ with top level state 0, we move back by $f^{-1}$, leaving the states unchanged.

The important point is that the $f,f^{-1}$ pair is fixed.
We use them to figure out what happens to the representative states, no matter which subset of the states we are in.
In the beginning we could choose it to be ${^+3}$, which would change the final decomposition, but it would be isomorphic  to the canonical odometer (just counting by $^+3$ instead of $^+1$).
The idea of compression appears here: we have a representative, but we need to know how to get to the original, and that's the coordinate value above.

As the bottom level component ends up having composition only partially defined,
it makes sense to switch to this type of structure entirely for the whole algorithm.
Compression has a neat formulation there as well in terms of isomorphic objects.
Therefore, the solution is to generalize from semigroups to semigroupoids.

\section{Semigroupoids and Relational Functors}
\label{sect:sgpoids}

We can approach semigroupoids from different fields by adding or removing structure.
  \begin{description}
    \item[category theory] A semigroupoid is a category without the condition for an identity arrow for each object.
    \item[graph theory] A semigroupoid is a directed graph where arrows can be composed into paths.
    \item[algebra] A semigroupoid is a semigroup with a restricted multiplication table representing the islands of composability.
    \item[computer programs] A semigroupoid describes how typed functions in a functional programming language can be combined.
  \end{description}
The formal definition follows the first approach.
\subsection{Semigroupoids and their Transformation Representations}
The following is the definition of a category with the need for the identity arrows removed.
\begin{definition}
A \emph{semigroupoid} $\cS$ consists of
\begin{itemize}
  \item a set of \emph{objects} $\Ob(\cS)$;
  \item a set of \emph{arrows} $\cS(X,Y)$ for each $X,Y\in\Ob(\cS)$, the
    so-called \emph{hom-set};
  \item a function called \emph{composition} for each $X,Y,Z\in Ob(\cS)$
  \begin{align*}
    \cS(X,Y)\times \cS(Y,Z) &\rightarrow \cS(X,Z) \\
    (f,g) &\mapsto fg
  \end{align*}
  satisfying the \emph{associativity} condition, i.e., if $f\in\cS(X,Y)$, $g\in\cS(Y,Z)$ and $h\in\cS(Z,W)$ then
  $$f(gh)=(fg)h.$$
\end{itemize}
\end{definition}

An arrow $f\in\cS(X,Y)$ has \emph{domain} $\dom(f)=X$ and \emph{codomain} $\cod(f)=Y$.
We can write this in concise notation as $f:X\rightarrow Y$ or  $X\stackrel{f}{\rightarrow}Y$.
As an alternative notation for $\cS(X,Y)$, $\Hom_\cS(X,Y)$ emphasizes that it is a collection of homomorphisms.
We will write $\ENDO{X}$ for $\cS(X,X)$ and $\vv{XY}$ for $\cS(X,Y)$.
The name of the semigroupoid itself, like $\cS$ or $\cT$, refers to the total set of arrows.

We call the set of objects the set of formal $types$.
They are constraints on composability, determining which transformation can be put in a sequence.
Types are arguably a form of decomposition by separation.
Example \ref{ex:dual-counter} shows how typing enables precise representation of a specified behaviour.
The \emph{type of an arrow} is its domain-codomain pair.

\begin{example}[Two-object semigroupoid] Fig.~\ref{2obj6arr} shows a semigroupoid with two objects $X,Y$. Morphisms $a,b$ are of type $\ENDO{X}$, $c,d,e$ are of type $\vv{XY}$, and $f$ is of type $\ENDO{Y}$.
  \label{ex:2obj-sgpoid}
\end{example}

\begin{figure}[ht]
\begin{center}
\begin{tblr}{Q[c,m]Q[c,m]Q[c,m]} 
\begin{tikzpicture}[shorten >=1pt, node distance=2cm, on grid, auto,inner sep=2pt]
    \node[draw, circle] (X)   {$X$};
    \node[draw, circle] (Y) [right of=X] {$Y$};
    \path[->,every loop/.append style=-{Latex}]
    (X) edge [arr, loop above] node {$a$} (X)
    (X) edge [arr, loop below] node {$b$} (X)
    (X) edge [arr] node {$d$} (Y)
    (X) edge [arr,bend left=42, above] node {$c$} (Y)
    (X) edge [arr,bend right=42, below] node {$e$} (Y)
    (Y) edge [arr,loop right] node {$f$} (Y);
\end{tikzpicture}
&
\begin{tblr}{
  hline{2-8}={2-8}{0.4pt,solid},
  vline{2-8}={2-8}{0.4pt,solid}}
    & $a$ & $b$ & $c$ & $d$ & $e$ & $f$ \\
$a$ & $a$ & $b$ & $c$ & $d$ & $e$ & \\
$b$ & $b$ & $a$ & $c$ & $e$ & $d$ & \\
$c$ &  &  &  &  &  &  $c$\\
$d$ &  &  &  &  &  &  $c$\\
$e$ &  &  &  &  &  &  $c$\\
$f$ &  &  &  &  &  &  $f$ \\
\end{tblr}
&
\begin{tblr}{
hline{2-5}={2-4}{0.4pt,solid},
vline{2-5}={2-4}{0.4pt,solid}}
           & $\ENDO{X}$ & $\vv{XY}$ & $\ENDO{Y}$  \\
$\ENDO{X}$ & $\ENDO{X}$ & $\vv{XY}$ &  \\
$\vv{XY}$ &  &  & $\ENDO{XY}$ \\
$\ENDO{Y}$ &  & & $\ENDO{Y}$ \\
\end{tblr}
\end{tblr}
\end{center}
\caption{A semigroupoid with two objects and six arrows. Diagram of objects and arrows on the left, the corresponding composition table in the middle, and the simplified composition table with the types only on the right.}
\label{2obj6arr}
\end{figure}
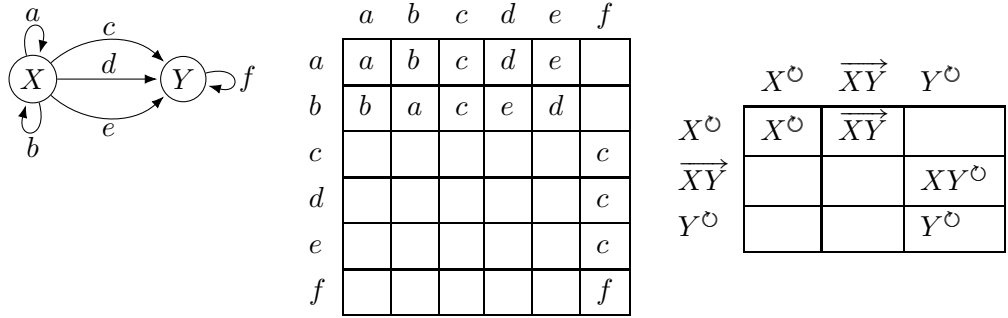

Analogously to the semigroup case, we can interpret semigroupoid arrows as
transformations of sets.
\begin{definition}[Transformation semigroupoid] A family of non-empty finite sets $X_i$, $i\in\{1,\ldots,n\}$ and transformations of type $X_i\rightarrow X_j$, $i,j\in\{1,\ldots,n\}$ closed under function composition form a transformation semigroupoid.
\end{definition}

In categorical language, we have a set-valued functor $F:\cS\rightarrow\Fin$.
The objects are mapped to sets and the arrows go to functions between those sets.
In this transformation representation, $\ENDO{X}F$ will be a set of endomorphisms of $XF$ ($F$ is assumed to be surjective onto hom-sets).
We can also call this set a \emph{stabilizer} of $XF$, as its endomorphisms form a semigroup.
They stabilize the set in the weak sense of not leaving it.
Similarly, we refer to $\vv{XY}F$
as the set of corresponding \emph{transporters} maps $XF\rightarrow YF$, but they don't form a semigroup as they are not composable by themselves.
We can imagine semigroupoids as a network of stabilizers connected by transporters.

Category theory can accommodate several levels of descriptions.
Sometimes, it can be challenging to know which level exactly are we at.
With transformation semigroupoids we have three levels.
The most abstract level consists of objects connected by composable arrows.
The composition table states the rules of composition without any explanation.
We can make this concrete by interpreting the objects as sets and the arrows as functions between them.
Then, the most specific is the level of state transitions, where we do the actual computations with the automata.
A central argument of this paper is that in order to understand state transition
systems it is enough to operate on the most abstract level.

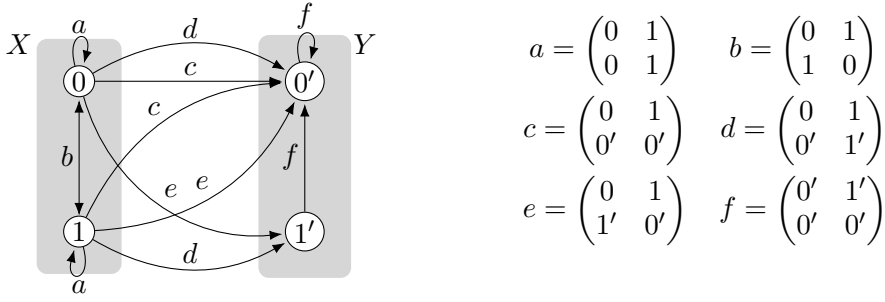
\begin{figure}[ht]
  \begin{subfigure}{.4\textwidth}
  \begin{tikzpicture}[shorten >=1pt, node distance=2cm, on grid, auto,inner sep=2pt]
    \tikzset{arr/.style={-Latex},
    rarr/.style={Latex-Latex},
    bg/.style={fill=gray!32,draw=none,rounded corners,inner sep=10pt},
    n/.style={draw,circle,inner sep=1pt,fill=white}}
    \node[n] (X0)   {$0$};
    \node[n] (X1) [below of=X0] {$1$};
    \node at (3,0) [n] (Y0)  {$0'$};
    \node[n] (Y1) [below of=Y0] {$1'$};
    \begin{pgfonlayer}{background}
      \node[fit=(X0) (X1), bg] (X) {};
    \end{pgfonlayer}
    \begin{pgfonlayer}{background}
      \node[fit=(Y0) (Y1), bg] (Y) {};
    \end{pgfonlayer}
    \node[] (X0label) [above right = 1.5cm and -.8cm of X] {$X$};
    \node[] (Y0label) [above right = 1.5cm and .8cm of Y] {$Y$};
    \path[->,every loop/.append style=arr,every edge/.append style=arr]
      (X0) edge [loop above] node {$a$} (X0)
      (X1) edge [loop below] node {$a$} (X1)
      (X1) edge [rarr] node {$b$} (X0)
      (Y0) edge [loop above] node {$f$} (Y0)
      (Y1) edge node {$f$} (Y0)
      (X0) edge node {$c$} (Y0)
      (X1) edge [bend left=30] node {$c$} (Y0)
      (X0) edge [bend left=30] node {$d$} (Y0)
      (X1) edge [bend right=30] node {$d$} (Y1)
      (X0) edge [bend right=40] node {$e$} (Y1)
      (X1) edge [bend right=30] node {$e$} (Y0);
\end{tikzpicture}
\end{subfigure}
\begin{subfigure}{.4\textwidth}
\begin{tblr}{cc}
$a=\left(\begin{matrix}0 & 1 \\ 0 & 1 \end{matrix}\right)$
&
$b=\left(\begin{matrix}0 & 1 \\1 & 0\end{matrix}\right)$ \\
$c=\left(\begin{matrix}0 & 1 \\ 0' & 0' \end{matrix}\right)$
&
$d=\left(\begin{matrix}0 & 1 \\0' & 1'\end{matrix}\right)$\\
$e=\left(\begin{matrix}0 & 1 \\ 1' & 0' \end{matrix}\right)$
&
$f=\left(\begin{matrix}0' & 1' \\0' & 0'\end{matrix}\right)$
\end{tblr}
\caption*{} 
\end{subfigure}
\caption{Transformation semigroupoid with stabilizers $a,b$ and $f$, and with
  transporters $c,d,e$.}
\label{fig:2obj6arrsgpoid}
\end{figure}

\begin{example}
We can give a transformation representation of the abstract 2-object semigroupoid in Example \ref{ex:2obj-sgpoid}.
$XF=\{0,1\}$ and $YF=\{0',1'\}$ and the transformations are listed in Fig.~\ref{fig:2obj6arrsgpoid}.
\end{example}

\emph{Generators} for semigroups are analogous to input symbols for finite
automata.
Composing them generates all possible dynamics.
For semigroupoids the situation is more involved.
In addition to the endoarrows, we need to consider all roundtrips from that
objects visiting other objects.
The process is reminiscent of the topological idea contracting loops (homotopy).
It is also the basic idea of the holonomy decomposition discussed in Section \ref{sect:holonomy}.

\subsection{Relational Functors}
To define the analogous structure preserving relations, we extend composition of arrows to sets of arrows:
$$ \{f_1,\ldots, f_n\}\{g_1,\ldots, g_m\}=\{ f_ig_j \mid f_i:X\rightarrow Y, g_j:Z\rightarrow W, Y=Z\},$$
i.e., we compose when we can, and ignore the rest.

\begin{definition}[Relational Functor]
  A \emph{relational functor} of semigroupoids  $\cS\xrightarrow{\varphi} \cT$ is a relation taking an $\cS$-arrow to a non-empty set of $\cT$-arrows, not necessarily of the same type.
  For all composable pairs of arrows $f:X\rightarrow Y$ , $g: Y\rightarrow Z$ two
  the condition of compatibility holds:
  $$\varphi_1(f)\varphi_1(g)\subseteq \varphi_1(fg),$$
  and there is at least one composite arrow in the target:
  $$ \varphi_1(f)\varphi_1(g)\neq\varnothing.$$
\end{definition}
In a sense, the arrows $\varphi_1(f)$ play the same role in $\cT$ as $f$ does in $\cS$.
The relation on arrows induces a relation on objects $\varphi_0: \Ob(\cS)\rightarrow \Ob(\cT)$.
 If there is a $\cT$-arrow with domain $X'$ related to an $\cS$-arrow with the domain $X$, then  $X\in\Ob(\cS)$ is related to $X'\in\Ob(\cT)$, and similarly for codomains.
Thus, according to the compatibility condition, the set of those arrows from
$\varphi_0(X)$ to $\varphi_0(Z)$  corresponding to the arrow $fg$ should include
all the composite arrows  factoring through $\varphi_0(Y)$, where the factors
correspond to arrows $f$ and $g$.
Factoring through a set of objects is a constraint, so there could be additional arrows in $\varphi_1(fg)$  that go some other way.
\begin{lemma}
Relational functors are composable.
\end{lemma}
\begin{proof}
 Let $\varphi: \cS\rightarrow\cT$ and $\tau:\cT\rightarrow\cU$ be relational functors.
 We need to show that $\tau(\varphi(f))\tau(\varphi(g))\subseteq\tau(\varphi(fg))$.
Let $f'\in\varphi(f)$ and $g'\in\varphi(g)$ be arbitrary picked arrows in
$\cT$.
Since, $\tau$ is a relational functor, $\tau(f')\tau(g')\subseteq\tau(f'g')$.
This will be true for the finite union of sets on both sides (going through all
possible $f'$ and $g'$).
\end{proof}

A relational functor $\varphi:\cS\rightarrow\cT$ is \emph{surjective} if
$\bigcup_{f\in\cS}\varphi(f)=\cT$, denoted by $\twoheadrightarrow$.
If $\varphi(f)\cap\varphi(g)\neq\varnothing\implies f=g$, then we have an
\emph{injective} relational functor, denoted by $\hookrightarrow$.
This is also called an \emph{emulation} or \emph{covering}.
It expresses that $\cT$ is at least as computationally powerful as $\cS$.
Since the image sets do not overlap, it is always well-defined what computation
of $\cS$ is represented in $\cT$.

\section{Cascade Decomposition of Semigroupoids}
\label{sect:sgpoiddec}

We apply the Covering Lemma method to (abstract) semigroupoids.
We only need to do pinhole projections for arrows, unlike in the case of transformation semigroups, where we need to make those projections for states and transformations.
First, we define a simple product and then apply compression to get the bottom level component, the generalized kernel.
Finally, we will prove that the hierarchical product resulting from compression emulates to original semigroupoid.

\subsection{The Tracing Product}
As a trivial limit case for the hierarchical product, we use a construction that is a subcategory of the direct product of categories.
The top level, the image of the surjective relation functor $\varphi$, gives a summary of, while the bottom level gives the complete dynamics of $\cS$.
The components are independent due to this redundancy.

\begin{definition}[Tracing product]
  Given a relational functor $\varphi:\cS\rightarrow \cT$, the \emph{tracing
    product} denoted by $\cT\times_{\varphi}\cS$, is the semigroupoid where the
  arrows are the elements of the graph of $\varphi^{-1}:\cT\rightarrow\cS$:
  $$\cT\times_{\varphi}\cS=\bigcup_{f\in\cT}\{f\}\times\varphi^{-1}(f).$$
Composition is done independently: $(f,a)(g,b)=(fg,ab)$ when both pairs, $f,g$ and $a,b$, are composable.
\end{definition}
In other words, $\cT$ is an annotation for $\cS$.
It gives a coarse-grained view, while on the bottom level we trace the actual path taken.
Composition is consistent as the levels are set up to be  synchronized.
The top level $\cT$ can be interpreted as the abstract specification of the correctness for $\cS$, $\varphi$ as the proof of correctness, and the tracing product as the tool for verification.
Not surprisingly, the tracing product can emulate the original semigroupoid.
\begin{lemma}
$\cS\hookrightarrow\cT\times_{\varphi}\cS$ for a surjective relational functor $\varphi:\cS\twoheadrightarrow \cT$.
\end{lemma}
\begin{proof}
Let $\tau$ be the relational functor projecting the tracing product to its second coordinate: $\tau(a)=\{(x,y)\mid y=a\}$, the set of all arrows in the product with second coordinate $a$.
It is injective, since $\tau(a)\cap\tau(b)\implies a=b$.

If $a$ and $b$ are composable in $\cS$, then $\tau(a)\tau(b)$ is guaranteed to have a composable pair, since $\varphi(a)\varphi(b)\neq\varnothing$.
We only need to check the top level, since the bottom levels are composable by assumption.

Similarly, $\tau(a)\tau(b)\subseteq\tau(ab)$ follows from $\varphi(a)\varphi(b)\subseteq\varphi(ab)$.
\end{proof}

\subsection{Compression and the Kernel of a Relational Functor}
When a pattern appears several times, it is enough to store it once and record the locations of the multiple occurrences.
We will use this principle of compression to reduce the size of a semigroupoid, whenever the same dynamics (set of arrows) appears several times.
First, we look at individual arrows, when two of them can be expressed by each other.
\begin{definition}
In a semigroupoid $\cS$ two arrows $f: X\rightarrow Y$ and $g: Z\rightarrow U$
are \emph{equivalent}, or \emph{interchangeable} if there exist arrows $\TO{X}{Z}$, $\TO{Z}{X}$,
$\TO{Y}{U}$,  and $\TO{U}{Y}$ such that $f=\TO{X}{Z}g\TO{U}{Y}$ and $g=\TO{Z}{X}f\TO{Y}{U}$, so that the following diagram commutes.
\begin{center}
\begin{tikzcd}
X \arrow[rrr, "f",arr]  \arrow[d, "\TO{X}{Z}"',arr, bend right=30]  & & & Y \arrow[d, "\TO{Y}{U}"',arr, bend right=30] \\
Z \arrow[rrr, "g"', arr]  \arrow[u, "\TO{Z}{X}"',arr, bend right=30] & & & U  \arrow[u, "\TO{U}{Y}"',arr, bend right=30]
\end{tikzcd}
\end{center}
Moreover, $\TO{X}{Z}\TO{Z}{X}f=f\TO{Y}{U}\TO{U}{Y}=f$ (and similarly for $g$), which is just a roundabout
way of saying (due to identity arrows lacking in semigroupoids) that the pairs of maps are \emph{inverses} of each other.

Any subsets of the four objects $X,Y,Z,U$ can be identified.
Equivalence is defined for endoarrows as well.
Indeed, for semigroups (as single object arrows) this relation is the same as the $\cD$-relation, one of the famous Green's relations in semigroup theory \cite{Howie95}.
\end{definition}

We extend this equivalence relation to sets of arrows (the preimages of $\varphi$).
For a set of arrows $P\subseteq\cS$, the set of objects that appear in $P$
as a domain or a codomain is denoted by $\Ob|_P(\cS)$.

\begin{definition}
The sets $P$ and $Q$ of arrows of $\cS$ are \emph{equivalent} if their
supporting sets of objects are isomorphic as sets, i.e., $\Ob|_P(\cS)\cong\Ob|_Q(\cS)$
in $\cS$.
More precisely, we have families of arrows
$\TO{P}{Q}:\Ob|_P(\cS)\rightarrow\Ob|_Q(\cS)$ and
$\TO{Q}{P}:\Ob|_Q(\cS)\rightarrow\Ob|_P(\cS)$ inducing two bijections in opposite directions between $P$ and
$Q$.
Moreover, the arrows corresponding to pairs of objects are inverses relative to
the arrows of $P\cup Q$.
\end{definition}

For a family of equivalent sets of arrows we pick a
representative naturally.
Natural means that the choice does not matter.
They all lead to the same construction.
We denote the representative by $\ast$, thus for a set of arrows $P$, we have
the arrows $\TOREP{P}$  to the representative, and arrows from the representative $\FROMREP{P}$.
Care must be taken for arrows with no other equivalent arrows.
These maps may not exist, since identity arrows are not guaranteed in semigroupoids.

\begin{definition}[Kernel of a semigroupoid relational functor]
Given a relational functor $\varphi:\cS\rightarrow \cT$, the \emph{kernel}
$\cK_\varphi$ is the semigroupoid $\cS$ with equivalent preimages identified.
Formally, $\cK_\varphi=\bigcup_{f\in\cT}[\varphi^{-1}(f)]$, where the square bracket indicates the equivalence class representative.
\end{definition}

The standard definition of the kernel is the preimage of the identity. That
works when it
is possible to recover all the other collapsings of the morphism (e.g., in
groups).
This kernel is a collection of all distinct collapsings of a given relational functor.
We need to know the preimages of all the arrows of $\cT$, except when there are
equivalent preimages, we only keep the representative.
In, general, $\cK_\varphi$ can be bigger than $\cS$ due to the overlapping image
sets when
$\varphi$ is not injective.
Note that the compression is not the same as the skeleton of the semigroupoid,
as not everything collapsible is collapsed by $\varphi$.

\subsection{Emulation by the Pinhole Cascade Product}
Now we can state and prove the main result: putting together the image of a surjective relational
functor with its kernel in a hierarchical way is just the compressed version of
the tracing product, thus it emulates the original semigroupoid.
\begin{definition}[Pinhole cascade product]
Given a relational functor $\varphi:\cS\rightarrow \cT$ the \emph{pinhole cascade product} $\cT\wr_{\varphi}\cK_\varphi$ is the
semigroupoid with the arrows
$\bigcup_{f\in\cT}\{f\}\times\left[\varphi^{-1}(f)\right]$.
Composition for $(f,a)$ and $(g,b)$ is defined as
$$(f,a)(g,b)= \left(fg,\FROMREP{fg}(\TOREP{f}a\FROMREP{f})(\TOREP{g}b\FROMREP{g})\TOREP{fg}\right).$$
For readability, we define the encoding process of taking  arrows of $\varphi^{-1}(f)$ to their representatives.
If the arrow $\square$ has at least one other equivalent arrow, then $\enc_f:\square\mapsto\FROMREP{f}\square\TOREP{f}$  otherwise, $\enc_f(\square)=\square$.
Decoding is defined similarly. When $\square'$ is interchangeable,
$\dec_f:\square'\mapsto \TOREP{f}\square'\FROMREP{f}$, if not interchangeable,
then $\dec_f(\square')=\square'$.
Encoding composed with decoding, and vice versa, yield the identity on the arrows in their scope.
Now the composition rule can be stated more succinctly:

$$(f,a)(g,b)= \left(fg,\enc_{fg}\left(\dec_f(a)\dec_g(b)\right)\right).$$
Note that the bottom level composition can also expressed as $\enc_{fg}(\dec_f(a))\enc_{fg}(\dec_g(b))$.
\end{definition}

The bottom level composition depends on the top level coordinates, but not the other way around.
Therefore, this embeds into the wreath product.
The only control information for passed for the top to bottom is to which `model' of the preimage to use.
Therefore, it is better not to start with the wreath product (as in most
mathematical texts), where all possible control signals appear.
If there is no compression, then composition in the top and bottom levels are
independent, then we only have a tracing product (see Example \ref{ex:nocompression}).

\begin{theorem}\label{thm:sgpoiddec}
If $\varphi: \cS\twoheadrightarrow\cT$ is a surjective relational functor then
the pinhole cascade product $\cT\wr_{\varphi}\cK_\varphi$ emulates $\cS$, i.e., $\cS\hookrightarrow\cT\wr_{\varphi}\cK_\varphi$.
\end{theorem}
\begin{proof}
We need to show that the uncompressed pinhole cascade is identical to the corresponding tracing product, $\cT\wr_{\varphi}\cK_\varphi=\cT\times_\varphi\cS$, by giving an identity bijection.
Let us pick two arrows $(f,a)$ and $(g,b)$ in the tracing product.
In the pinhole cascade, the lower level coordinates need to be encoded.
The encoded coordinate pairs are $(f,\enc_f(a))$ and $(g,\enc_g(b))$.
After composition, we have $(fg, \enc_{fg}(\dec_f(\enc_f(a))\dec_g(enc_g(b))))$.
The matching encoding-decoding pairs cancel.
The result is $(fg, \enc_{fg}(ab))$,
assuming $ab$ is a composite arrow in $\cS$.
After decoding, we have $(fg,ab)$.
\end{proof}
To summarize, the coordinate value arrows are encoded, but composition takes
place in the original semigroupoid, so they have to be decoded.
The usability of the decomposition depends on how easy is to interpret the
chosen representative.
In our positional number notation system, which is a cascade product of
$\Z_{10}$'s, they are particularly meaningful (see Example \ref{ex:statelessZ4}
for the binary case).
A remaining challenge for the applications of algebraic automata theory is the
construction of simple hierarchical rule tables for more general discrete
dynamical systems.

\section{Holonomy Decomposition}
\label{sect:holonomy}
In Krohn-Rhodes theory, the holonomy method for cascade decomposition was originally
developed by H. Paul Zeiger \cite{zeiger67a,zeiger68}, and subsequently improved by
S. Eilenberg \cite{eilenberg} (one of the founders of category theory), and later by several others \cite{automatanetworks2005,ginzburg_book68,
holcombe_textbook,Maler2010}.
It is also defined  for categories \cite{KRTforCategories}, closely following the transformation representation.

The term `holonomy' is borrowed from differential geometry, since a roundtrip of composed bijective maps producing permutations is analogous to moving a vector via parallel transport along a smooth closed curve yielding change of the angle of the vector.

The holonomy decomposition is defined for transformation representations.
When we forget that interpretation, its description becomes similar to the
Covering Lemma.
\begin{definition}The set $\imgs_S(X)=\{X\cdot s\mid s\in S\}$ is the \emph{image set} of the transformation semigroup $(X,S)$.
\end{definition}
The holonomy method works by the detailed examination of how $S$ acts on
$\imgs_S(X)$ by \emph{considering each image set as a separate type}.
Thus, we have a semigroupoid with objects $\imgs_S(X)$ and arrows defined by the
elements of $S$ restricted to the image sets.
This semigroupoid is potentially lot bigger than the original semigroup (see
Example \ref{ex:Tn}).
Two sets are equivalent if they have bijections between them in $S$.
This is where compression comes in.
The encoding and decoding processes are the same.

The algorithm differs as holonomy aims for the highest resolution decomposition
complete with all possible compressions. Hence the need for the subduction
relation, tiling, height and depth calculations.
Also, the holonomy decomposition requires a specific surjective relational morphism to start with.
We map states to subsets $x\mapsto X\setminus \{x\}$, permutations to themselves and any other transformations to constant maps to states not in their images \cite{FromRelToEmulationNehaniv1996}.
To have an iterative holonomy decomposition, these type of morphisms need to be generalized to relational functors of semigroupoids.

The proofs also have different styles.
The holonomy method constructs an elaborate decomposition, and then proves the emulation.
The Covering Lemma starts from the emulation and uses it as a constraint
and works out the details.

\section{Conclusions and Further Research Directions}
With the intention of fixing issues in transformation semigroup decomposition methods, we formulated the core algorithm on the more abstract, categorical level.
This yielded three distinct results:
\begin{enumerate}
  \item The Covering Lemma decomposition algorithm defined for semigroupoids has no deficiencies: it produces a well-defined bottom level component which is also a semigroupoid and thus iteration is possible.
  \item We deepened our understanding of hierarchical decompositions of computation by identifying three conceptual steps: collapse, copy, compress. This, in turn, showed that the only feedforward control signal is the position in the equivalence class. Therefore, we do not need to work with the combinatorially explosive wreath product.
  \item The abstract algorithm allows hierarchical decompositions of computation in a representation independent way, generalizing traditional state transition systems to other forms of computation.
\end{enumerate}
All these open up new directions for research.
Possibly the biggest impact is due to the representation agnostic decomposition algorithm.
We can take all finite diagram semigroups (i.e., subsemigroups of the partitioned binary relations \cite{Martin_Mazorchuk_2013}, including binary relations, partial transformations and permutations, Brauer monoids, and Temperley-Lieb/Jones monoids), and we now have a free Krohn-Rhodes Theorem \cite{primedecomp65} type of decomposition for each representations.
Instead of working out the details for each, we can do the decompositions abstractly, and see how they are realized in that particular form of computation.

Semigroups model computation by emphasizing composition.
Semigroupoids introduce type and see abstract computation as a network between the islands of composability.
This typed view reshapes some old problems.

Understanding is provided by decompositions, but the decomposition algorithm we described requires a surjective morphism to start with.
Now, we have the question: `What are the relational functors from a semigroupoid?'.
A systematic description is needed to make the holonomy algorithm iterative too.
In holonomy, we need the smallest collapsing morphic relations in order to produce the highest resolution decompositions.

Given an abstract semigroupoid, what are its minimal degree transformation representations?
This question is already far from trivial for semigroups (see e.g., \cite{CameronEFMPQ23}).
At least, we can have the right regular representation to start with if we have an identity for the
monoids.
For semigroupoids, we need to find other algorithms to produce those representations.

\bibliographystyle{plainnat}
\bibliography{../coords.bib}

\begin{thebibliography}{26}
\providecommand{\natexlab}[1]{#1}
\providecommand{\url}[1]{\texttt{#1}}
\expandafter\ifx\csname urlstyle\endcsname\relax
  \providecommand{\doi}[1]{doi: #1}\else
  \providecommand{\doi}{doi: \begingroup \urlstyle{rm}\Url}\fi

\bibitem[Awodey(2010)]{awodey2010category}
S.~Awodey.
\newblock \emph{Category Theory}.
\newblock Oxford Logic Guides. OUP Oxford, 2010.
\newblock ISBN 9780199587360.

\bibitem[Barr and Wells(1995)]{barr1995category}
M.~Barr and C.~Wells.
\newblock \emph{Category Theory for Computing Science}.
\newblock Number v. 1 in Prentice-Hall international series in computer
  science. Prentice Hall, 1995.
\newblock ISBN 9780133238099.

\bibitem[Cameron et~al.(2023)Cameron, East, Fitzgerald, Mitchell, Pebody, and
  Quinn{-}Gregson]{CameronEFMPQ23}
Peter~J. Cameron, James East, Desmond~G. Fitzgerald, James~D. Mitchell, Luke
  Pebody, and Thomas Quinn{-}Gregson.
\newblock Minimum degrees of finite rectangular bands, null semigroups, and
  variants of full transformation semigroups.
\newblock \emph{Comb. Theory}, 3\penalty0 (3), 2023.
\newblock \doi{10.5070/C63362799}.

\bibitem[D{\"o}m{\"o}si and Nehaniv(2005)]{automatanetworks2005}
P{\'a}l D{\"o}m{\"o}si and Chrystopher~L. Nehaniv.
\newblock \emph{{Algebraic Theory of Finite Automata Networks: An
  Introduction}}, volume~11 of \emph{{SIAM Series on Discrete Mathematics and
  Applications}}.
\newblock Society for Industrial and Applied Mathematics, 2005.

\bibitem[Egri-Nagy et~al.(2024)Egri-Nagy, Nehaniv, and Mitchell]{SgpDec}
A.~Egri-Nagy, C.~L. Nehaniv, and J.~D. Mitchell.
\newblock \emph{{\textsc{{S}gp{D}ec} -- software package for Hierarchical
  Composition and Decomposition of Permutation Groups and Transformation
  Semigroups, Version 1.1.0}}, 2024.
\newblock \url{https://gap-packages.github.io/sgpdec/}.

\bibitem[Egri-Nagy and Nehaniv(2024)]{egrinagy2024relation}
Attila Egri-Nagy and Chrystopher~L. Nehaniv.
\newblock From relation to emulation and interpretation: Computer algebra
  implementation of the covering lemma for finite transformation semigroups.
\newblock In Szil{\'a}rd~Zsolt Fazekas, editor, \emph{Implementation and
  Application of Automata}, pages 138--152, Cham, 2024. Springer Nature
  Switzerland.
\newblock ISBN 978-3-031-71112-1.
\newblock \doi{10.1007/978-3-031-71112-1_10}.

\bibitem[Eilenberg(1976)]{eilenberg}
Samuel Eilenberg.
\newblock \emph{{Automata, Languages and Machines, vol.\ B}}.
\newblock Academic Press, 1976.

\bibitem[GAP()]{GAP4}
GAP.
\newblock \emph{{GAP -- Groups, Algorithms, and Programming, Ver.\ 4.14.0}}.
\newblock The GAP~Group, 2024.
\newblock URL \url{https://www.gap-system.org}.

\bibitem[Ginzburg(1968)]{ginzburg_book68}
Abraham Ginzburg.
\newblock \emph{{Algebraic Theory of Automata}}.
\newblock Academic Press, 1968.

\bibitem[Holcombe(1982)]{holcombe_textbook}
W.~M.~L. Holcombe.
\newblock \emph{{Algebraic Automata Theory}}.
\newblock Cambridge University Press, 1982.

\bibitem[Howie(1995)]{Howie95}
John~M. Howie.
\newblock \emph{{Fundamentals of Semigroup Theory}}, volume~12 of \emph{{London
  Mathematical Society Monographs New Series}}.
\newblock Oxford University Press, 1995.

\bibitem[Krohn and Rhodes(1965)]{primedecomp65}
Kenneth Krohn and John Rhodes.
\newblock {Algebraic Theory of Machines. {I}. {P}rime Decomposition Theorem for
  Finite Semigroups and Machines}.
\newblock \emph{Transactions of the American Mathematical Society},
  116:\penalty0 450--464, April 1965.

\bibitem[Lawvere and Rosebrugh(2003)]{lawvere2003sets}
F.W. Lawvere and R.~Rosebrugh.
\newblock \emph{Sets for Mathematics}.
\newblock Cambridge University Press, 2003.
\newblock ISBN 9780521010603.

\bibitem[Lawvere and Schanuel(2009)]{lawvere2009conceptual}
F.W. Lawvere and S.H. Schanuel.
\newblock \emph{Conceptual Mathematics: A First Introduction to Categories, 2nd
  edition}.
\newblock Cambridge University Press, 2009.
\newblock ISBN 9780521894852.

\bibitem[Leinster(2014)]{Leinster_2014}
Tom Leinster.
\newblock \emph{Basic Category Theory}.
\newblock Cambridge Studies in Advanced Mathematics. Cambridge University
  Press, 2014.
\newblock \doi{10.1017/CBO9781107360068}.
\newblock URL \url{https://arxiv.org/abs/1612.09375}.

\bibitem[MacLane(1998)]{MacLane1998}
Saunders MacLane.
\newblock \emph{Categories for the Working Mathematician (Graduate Texts in
  Mathematics)}.
\newblock Springer, Berlin, 2nd edition, 10 1998.
\newblock ISBN 9780387984032.

\bibitem[Maler(2010)]{Maler2010}
Oded Maler.
\newblock On the {K}rohn-{R}hodes cascaded decomposition theorem.
\newblock In Zohar Manna and Doron~A. Peled, editors, \emph{Time for
  Verification}, pages 260--278. Springer-Verlag, Berlin, Heidelberg, 2010.
\newblock ISBN 3-642-13753-9, 978-3-642-13753-2.

\bibitem[Martin and Mazorchuk(2013)]{Martin_Mazorchuk_2013}
Paul Martin and Volodymyr Mazorchuk.
\newblock Partitioned binary relations.
\newblock \emph{MATHEMATICA SCANDINAVICA}, 113\penalty0 (1):\penalty0 30–52,
  Sep. 2013.
\newblock \doi{10.7146/math.scand.a-15480}.

\bibitem[Nehaniv(1996)]{FromRelToEmulationNehaniv1996}
Chrystopher~L. Nehaniv.
\newblock {From relation to emulation: The Covering Lemma for transformation
  semigroups}.
\newblock \emph{Journal of Pure and Applied Algebra}, 107\penalty0
  (1):\penalty0 75--87, 1996.
\newblock \doi{10.1016/0022-4049(95)00030-5}.

\bibitem[Rhodes(2009)]{wildbook}
John Rhodes.
\newblock \emph{{Applications of Automata Theory and Algebra via the
  Mathematical Theory of Complexity to Biology, Physics, Psychology,
  Philosophy, and Games}}.
\newblock World Scientific Press, 2009.
\newblock Foreword by Morris W.\ Hirsch, edited by Chrystopher L.\ Nehaniv
  (Original version: UC Berkeley, Mathematics Library, 1971).

\bibitem[Rhodes and Steinberg(2008)]{QBook}
John Rhodes and Benjamin Steinberg.
\newblock \emph{{The q-theory of Finite Semigroups}}.
\newblock Springer, 2008.

\bibitem[Riehl(2017)]{riehl2017category}
E.~Riehl.
\newblock \emph{Category theory in context}.
\newblock Dover Publications, 2017.
\newblock ISBN 978-0-486-82080-4.

\bibitem[Tilson(1987)]{TilsonCat}
Bret~R. Tilson.
\newblock {Categories as Algebras: An Essential Ingredient in the Theory of
  Monoids}.
\newblock \emph{Journal of Pure \& Applied Algebra}, 48:\penalty0 83--198,
  1987.

\bibitem[Wells(1980)]{KRTforCategories}
Charles Wells.
\newblock A {K}rohn-{R}hodes theorem for categories.
\newblock \emph{Journal of Algebra}, 64:\penalty0 37--45, 1980.
\newblock \doi{10.1016/0021-8693(80)90130-1}.

\bibitem[Zeiger(1967)]{zeiger67a}
H.~Paul Zeiger.
\newblock {Cascade synthesis of finite state machines}.
\newblock \emph{Information and Control}, 10\penalty0 (4):\penalty0 419--433,
  1967.
\newblock Erratum: {\bf 11}(4): 471 (1967).

\bibitem[Zeiger(1968)]{zeiger68}
H.~Paul Zeiger.
\newblock {Cascade Decomposition Using Covers}.
\newblock In Michael~A. Arbib, editor, \emph{{Algebraic Theory of Machines,
  Languages, and Semigroups}}, chapter~4, pages 55--80. Academic Press, 1968.

\end{thebibliography}

\appendix

\section{Additional Examples}
\begin{example}[Dual-mode counter]
\label{ex:dual-counter}
We want to build a machine with two distinct modes.
One for a binary counter, the other one for a mod-3 counter.
Two switch operations change the modes and reset the counter upon switching.
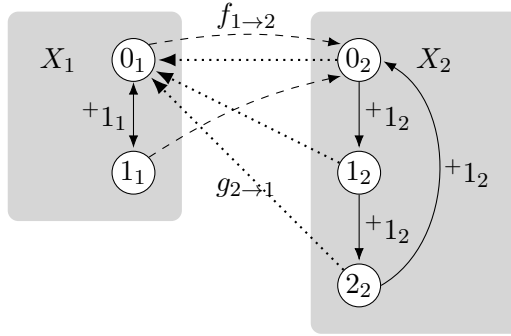
\begin{figure}[h]
\begin{tikzpicture}[shorten >=1pt, node distance=1.5cm, on grid, auto,inner sep=2pt]
\tikzset{arr/.style={-Latex},
         rarr/.style={Latex-Latex},
         bg/.style={fill=gray!32,draw=none,rounded corners,inner sep=10pt},
         n/.style={draw,circle,inner sep=1pt,fill=white}}
    \node[n] (X10) {$0_1$};
    \node[n] (X11) [below of=X10] {$1_1$};
    \node at (3,0) [n] (X20)  {$0_2$};
    \node[n] (X21) [below of=X20] {$1_2$};
    \node[n] (X22) [below of=X21] {$2_2$};
    \node[] (X1label) [left = 1cm of X10] {$X_1$};
    \node[] (X2label) [right = 1cm of X20] {$X_2$};
    \path[->,every loop/.append style=arr,every edge/.append style=arr]
      (X10) edge [rarr] node[left] {$^+1_1$} (X11)
      (X20) edge node {$^+1_2$} (X21)
      (X21) edge node {$^+1_2$} (X22)
      (X22.east) edge [bend right=60] node[right] (label) {$^+1_2$} (X20.east)
      (X10.north east) edge [dashed, bend left=10] node {$f_{1\rightarrow 2}$} (X20.north west)
      (X11.north east) edge [dashed, bend left=10] (X20.south west)
      (X20) edge[dotted,thick] (X10)
      (X21) edge[dotted,thick] (X10)
      (X22) edge[dotted,thick] node [below] {$g_{2\rightarrow 1}$}(X10);
      \begin{pgfonlayer}{background}
        \node[fit=(X20) (X22) (label), bg] {};
      \end{pgfonlayer}
      \begin{pgfonlayer}{background}
        \node[fit=(X10) (X11) (X1label), bg] {};
      \end{pgfonlayer}
  \end{tikzpicture}
  \caption{Dual-mode 2-3-counter transformation semigroupoid.
  Only generator transformations shown. Dashed and dotted arrows are used to avoid excessive labelling.}
\end{figure}
The semigroupoid has two types $X_1$ and $X_2$, represented as sets $\{0_1,1_1\}$ and $\{0_2,1_2,2_2\}$.
The generator transformations are
$$^+1_1=\left(\begin{matrix}
  0_1 & 1_1 \\
  1_1 & 0_1
\end{matrix}\right),
^+1_2=\left(\begin{matrix}
  0_2 & 1_2 & 2_2\\
  1_2 & 2_2 & 0_2
\end{matrix}\right),
f_{1\rightarrow 2}=\left(\begin{matrix}
  0_1 & 1_1 \\
  0_2 & 0_2
\end{matrix}\right),
g_{2\rightarrow 1}=\left(\begin{matrix}
  0_2 & 1_2 & 2_2\\
  0_1 & 0_1 & 0_1
\end{matrix}\right).
$$
There is an identity for type $X_1$ by $(^+1_1)^2=^+0_1$.
In $X_2$, $^+1_2$ generates $^+2_2$ and $^+0_2$, the identity for $X_2$.
The switching maps are constant functions, therefore they do not transfer any of the permutation dynamics, they only create the constant maps.
The semigroupoid has 15 transformations.

Can we represent these transformations as a transformation semigroup?
There are 5 states in total, thus we can try to embed it into the full transformation semigroup $T_5$ by padding with identities.
However, this would generate more elements.
The composition of two generator counting operations would produce the
transformation
$^+1_1^+1_2=\left(\begin{smallmatrix}0_1 & 1_1 & 0_2 & 1_2 & 2_2\\
  1_1&0_1& 1_2 & 2_2 &0_2\end{smallmatrix}\right)$.
By taking powers,this generates an orbit with 6 elements: $^+1_1^+1_2$, $^+0_1^+2_2$,
$^+1_1^+0_2$,$^+0_1^+1_2$, $^+1_1^+2_2$, and $^+0_1^+0_2$.
However, we did not design the machine to count up to 6.

We can try to embed into $T_6$ by adding a sink state representing partial transformations.
Beyond adding one more state, this construction would allow `crashing' the
machine by providing invalid (untyped) inputs.
Again, by design, we may need to avoid this behaviour.
Using the stereotypical example of a vending machine, in most cases we do not
want to have a sequence of operations that takes the machine into an inescapable
useless state.

Compared to untyped transformation semigroups, transformation semigroupoids allow more precise and efficient expressions of computing structures.
By increasing the number of states, we can always find an equivalent untyped representation, thus they have the same computational power (in terms of recognizing languages).
\end{example}

\begin{example}[Communicating vessels -- Transferring dynamics]
Objects connected by isomorphisms, and thus sets of states with bijective maps between have the same semigroup.
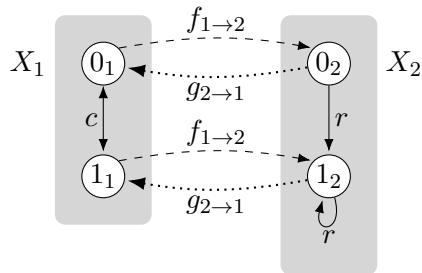
\begin{figure}[h]
  \begin{tikzpicture}[shorten >=1pt, node distance=1.5cm, on grid, auto,inner sep=2pt]
    \tikzset{arr/.style={-Latex},
             rarr/.style={Latex-Latex},
             bg/.style={fill=gray!32,draw=none,rounded corners,inner sep=10pt},
             n/.style={draw,circle,inner sep=1pt,fill=white}}
    \node[n] (X10)   {$0_1$};
    \node[n] (X11) [below of=X10] {$1_1$};
    \node at (3,0) [n] (X20)  {$0_2$};
    \node[n] (X21) [below of=X20] {$1_2$};
    \node[] (X1label) [left = 1cm of X10] {$X_1$};
    \node[] (X2label) [right = 1cm of X20] {$X_2$};
    \path[->,every loop/.append style=arr,every edge/.append style=arr]
      (X10) edge [rarr] node[left] {$c$} (X11)
      (X20) edge node {$r$} (X21)
      (X21) edge [loop below] node (lab) {$r$} (X21)
      (X10.north east) edge [dashed, bend left=10] node {$f_{1\rightarrow 2}$} (X20.north west)
      (X11.north east) edge [dashed, bend left=10] node {$f_{1\rightarrow 2}$} (X21.north west)
      (X20) edge[dotted,thick,bend left=10] node {$g_{2\rightarrow 1}$} (X10)
      (X21) edge[dotted,thick,bend left =10] node {$g_{2\rightarrow 1}$}(X11);
      \begin{pgfonlayer}{background}
        \node[fit=(X20) (X21) (lab), bg] {};
      \end{pgfonlayer}
      \begin{pgfonlayer}{background}
        \node[fit=(X10) (X11), bg] {};
      \end{pgfonlayer}
  \end{tikzpicture}
\caption{Type $X_1$ has a transposition, and type $X_2$ has reset.
The connecting $f,g$ transformations transfer these, so both objects end up with the same permutation-reset automaton.}
\end{figure}
With $c=\left( \begin{smallmatrix}
  0_1 & 1_1 \\ 1_1 & 0_1
\end{smallmatrix}\right)$, $r=\left( \begin{smallmatrix}
  0_2 & 1_2 \\ 1_2 & 1_2
\end{smallmatrix}\right)$,
$f_{1\rightarrow 2}=\left( \begin{smallmatrix}
  0_1 & 1_1 \\ 0_2 & 1_2
\end{smallmatrix}\right)$,
$g_{2\rightarrow 1}=\left( \begin{smallmatrix}
  0_2 & 1_2 \\ 0_1 & 1_1
\end{smallmatrix}\right)$ we can see how the action is transferred.
The cycle goes from $X_1$ to $X_2$ by $g_{2\rightarrow 1}cf_{1\rightarrow 2}=\left( \begin{smallmatrix}
  0_2 & 1_2 \\ 1_2 & 0_2
\end{smallmatrix}\right)$.
The reset goes from $X_2$ to $X_1$ by $f_{1\rightarrow 2}rg_{2\rightarrow 1}=\left( \begin{smallmatrix}
  0_1 & 1_1 \\ 1_1 & 1_1
\end{smallmatrix}\right)$.

Similar transfers explain how bigger groups can be assembled.
For example, if a type realized by a three-element set with a 3-cycle only can generate $S_3$ if it has a bijective map to a two-element set with a transposition.
\end{example}

\begin{example}[Stateless representation of $\Z_4$ built hierarchically from two
  $\Z_2$'s]
  \label{ex:statelessZ4}
Without a doubt, the hierarchical combination of counters is very familiar to us, since our number notation system taught in school works the same way.
It is a particularly well-behaving example of a wreath product, thus we can give a description without any reference to states.
We need to provide an operation whose powers form a 4-cycle.
\end{example}
\setlength{\intextsep}{0pt}%
\setlength{\columnsep}{3pt}%
\begin{wrapfigure}[3]{r}{.2\textwidth}
  \centering
  \begin{tikzpicture}[node distance=14pt]
    \tikzset{obj/.style={fill,circle,inner sep=1.5pt},
            arr/.style={-Latex}}
    \node [obj] (o) {};
    \node [obj,below of=o] (b) {};
    \path[->,every loop/.append style=arr]
    (o) edge [loop left] node {$^+0_1$} (o)
    (o) edge [loop right] node {$^+1_1$} (o)
    (b) edge [loop left] node {$^+0_2$} (b)
    (b) edge [loop right] node {$^+1_2$} (b);
  \end{tikzpicture}
\end{wrapfigure}
\vskip-1.4ex
We have two $\Z_2$ components.
The top level (depth 1) counts the 1's, the bottom level (depth 2) counts the 2's.
We take the direct product of the arrows, so we have 4 `coordinatized' operations: $\left\{ (^+0_1,^+0_2),(^+1_1,^+0_2),(^+1_1,^+1_2),(^+0_1,^+1_2)\right\}$. But how can we compose these?

The idea of the wreath product is  we compose in the top level without
considering anything else, so we have the composition table of $\Z_2$. What we
do on the bottom level, depends on what happens on the top. In the wreath
product, we use transformation representations, thus we can use the state to
determine on the top level.\hfill\null 
\begin{wrapfigure}[5]{r}{.2\textwidth}
  \centering
  \begin{tblr}{
    hline{2-4}={2-4}{0.4pt,solid},
    vline{2-4}={2-4}{0.4pt,solid}}
   & $^+0_1$ & $^+1_1$\\
  $^+0_1$ & $^+0_2$ & $^+0_2$\\
  $^+1_1$ & $^+0_2$ & $^+1_2$
  \end{tblr}
\end{wrapfigure}
Here we use a \emph{rule table}, connecting the composition table of the top level to the bottom level's.
The rule is very simple: do not do anything unless there is a carry bit in the case of composing $^+1_1$ by itself.
The operation $(^+0_1,^+0_2)$ is the identity and $c=(^+1_1,^+0_2)$ is the generating increase by 1 operation.
Thus, $c^2=(^+0_1,^+1_2)$, $c^3=(^+1_1,^+1_2)$ and $c^4$ is the identity. In short, $c$ represents $^+1$ in $\Z_4$.
If we don't do compression, then we get an 8-cycle. We can still use the same rule table though.

\begin{example}[Full Transformation Semigroup to Power Set Action]
  \label{ex:Tn}
Let's denote the $n$-element set $\{1,\ldots,n\}$ by $X$.
The full transformation semigroup $T_n$ consists of all transformations of type $X\rightarrow X$, thus $|T_n|=n^n$.
The image set has all the subsets except the empty set: $\imgs_{T_n}(X)=2^X\setminus\{\varnothing\}$, $2^n-1$ subsets.

Let's denote $\cT$ as the one-object semigroupoid with the arrows of the transformations in $T_n$, and $\cI$ the semigroupoid with objects $\imgs_{T_n}(X)$ and with the arrows of all the possible transformations between them (not the elements of $T_n$, but the transformations they induce). To simplify notation, we will not explicitly write the canonical set-valued functor, just use the sets for the objects.

Now we construct the relational functor $\varphi: \cT\rightarrow \cI$.
The single object $X$ of $\cT$ goes to the complete set of objects $2^X\setminus\{\varnothing\}$ of $\cI$.
Similarly, an arrow $t$ in $\cT$ goes to $2^n-1$ arrows, one arrow from each object of $\cI$.
They are all in different hom-sets (since the domains are different), according to how $t$ acts on the subsets of $X$:
$$ t:X\rightarrow X \mapsto \left\{\iota:X_i\rightarrow X_j  \mid X_it=X_j, t|_{X_i}=\iota  \right\}.$$
We need to check the compatibility condition: $\varphi_1(t)\varphi_1(t')\subseteq\varphi_1(tt')$.
In $\cT$ we can always compose.
The set of arrows $\varphi_1(t)$ will have exactly one arrow from each subset $X_i$, but no matter where they end, there will be exactly one arrow from $\varphi_1(t')$ to compose with. After the set-wise composition we will have exactly one arrow from each object corresponding to $tt'$, and this is exactly the set $\varphi_1(tt')$. We have equality of the sets of arrows.

The functor $\varphi$ is full (surjective on objects and hom-sets), and goes from a single-object to a many-objects semigroupoid.
In a sense we can divide a type into several subtypes.
This construction is also the base of the holonomy decomposition.
\end{example}

\begin{example}[Collapsing relational functor, but no compression yields tracing
  product.] We define $\varphi$ by $\varphi_0(0_1)=\varphi_0(1_1)=\{S_1\}$,
  $\varphi_0(0_2)=\varphi_0(1_2)=\{S_2\}$, and  $\varphi_1(a)=\{h\}$, $\varphi_1(b)=\{f\}$,
$\varphi_1(d)=\{g\}$,
$\varphi_1(c)=\{k\}$.

  \label{ex:nocompression}
 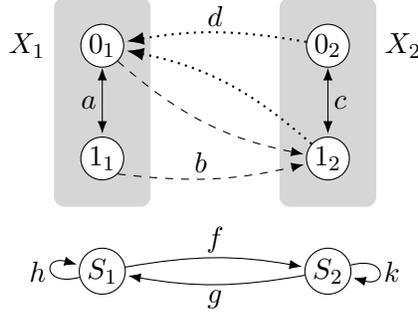
\begin{figure}[h]
  \begin{tikzpicture}[shorten >=1pt, node distance=1.5cm, on grid, auto,inner sep=2pt]
    \tikzset{arr/.style={-Latex},
             rarr/.style={Latex-Latex},
             bg/.style={fill=gray!32,draw=none,rounded corners,inner sep=10pt},
             n/.style={draw,circle,inner sep=1pt,fill=white}}
    \node[n] (X10)   {$0_1$};
    \node[n] (X11) [below of=X10] {$1_1$};
    \node at (3,0) [n] (X20)  {$0_2$};
    \node[n] (X21) [below of=X20] {$1_2$};
    \node[] (X1label) [left = 1cm of X10] {$X_1$};
    \node[] (X2label) [right = 1cm of X20] {$X_2$};
    \node[n] (S1) [below of=X11] {$S_1$};
    \node[n] (S2) [below of=X21] {$S_2$};
    \path[->,every loop/.append style=arr,every edge/.append style=arr]
    (S1) edge [bend left=10] node {$f$} (S2)
    (S2) edge [bend left=10] node {$g$} (S1)
    (S1) edge [loop left] node {$h$} (S1)
     (S2) edge [loop right] node {$k$} (S2)
     (X10) edge [rarr] node[left] {$a$} (X11)
      (X20) edge [rarr] node {$c$} (X21)
      (X10.south east) edge [dashed, bend right=15]  (X21)
      (X11.south east) edge [dashed, bend right=10] node {$b$} (X21)
      (X20) edge[dotted,thick,bend right=10] node[above] {$d$} (X10)
      (X21) edge[dotted,thick,bend right=15] (X10);
      \begin{pgfonlayer}{background}
        \node[fit=(X20) (X21), bg] {};
      \end{pgfonlayer}
      \begin{pgfonlayer}{background}
        \node[fit=(X10) (X11), bg] {};
      \end{pgfonlayer}
  \end{tikzpicture}
\caption{In semigroupoid above both types $X_1$and $X_2$ have the same $\Z_2$ dynamics.
However, the connecting transporters are constant maps, therefore there are no
interchangeable arrows.
When mapped onto the semigroupoid below, there is no compression.}
\end{figure}
The other generated transformations map similarly. The constants
$\varphi_1(da)=\{g\}$, $\varphi_1(bc)=\{f\}$, and identities map to the
corresponding loops.
Therefore, we have collapsings by $\varphi$.
The monoids for types $X_1$ and $X_2$ are isomorphic, but they are not
interchangeable in the semigroupoid, as there are no invertible bijections
between them.
Thus, there is no compression and the pinhole cascade product reverts to the
tracing product. The two $\Z_2$'s are kept separate.

\end{example}

\end{document}